\documentclass[12pt,psfig,reqno]{amsart}

\usepackage{amsfonts}
\usepackage{amscd}
\usepackage{amssymb,amsfonts,latexsym}
\usepackage{graphics,verbatim}
\usepackage{graphicx}
\usepackage{indentfirst}
\usepackage[usenames]{color} 
\usepackage[noadjust]{cite}
\usepackage{hyperref}

\makeatletter

\renewcommand\section{\@startsection {section}{1}{\z@}%
                                   {-3.5ex \@plus -1ex \@minus -.2ex}%
                                   {2.3ex \@plus.2ex}%
                                   {\centering\normalfont\bf}}

 \numberwithin{equation}{section}
\numberwithin{equation}{section}

\numberwithin{equation}{section}

\theoremstyle{plain}
\newtheorem{thm}{Theorem}[section]
\newtheorem{lemma}[thm]{Lemma}

\newtheorem{cor}[thm]{Corollary}
\newtheorem{ex}[thm]{Example}
\newtheorem{de}[thm]{Definition}

\begin{document}

\title{Spectrality and non-spectrality of a class of Moran measures with three-element digits}

\author{Xiao-Yu Yan}
\author {Wen-Hui Ai $^*$}

\date{\today}
\thanks{This work was supported by the NNSF of China (Nos. 12201206 and 12371072) and Hunan Provincial Natural Science Foundation of China (No. 2024JJ6301).\\
$^*$Corresponding author.}
\address{Key Laboratory of Computing and Stochastic Mathematics (Ministry of Education), School of Mathematics and Statistics, Hunan Normal University, Changsha, Hunan 410081, P.R. China.}

\email{xyyan1103@163.com}
\email{awhxyz123@163.com}
\maketitle \vskip 0.05in

\begin{abstract}
A Borel probability measure \( \mu \) with compact support on \( \mathbb{R}^n \) is called spectral measure if there exists a discrete set \( \Lambda \subset \mathbb{R}^n \) such that \( E_\Lambda := \{e^{2\pi i \langle \lambda, x \rangle}: \lambda \in \Lambda\} \) forms an orthonormal basis of \( L^2(\mu) \). In this paper, we study the spectrality and non-spectrality of a class of Moran measures with three-element digits on \( \mathbb{R} \).
Let
$p_n\in 3\mathbb Z\setminus\{0\}$ and $\mathcal{D}_n=\{0,a_n,b_n\}$ with $\{a_n,b_n\}=\{-1,1\}\pmod 3$.
It is know that the infinite convolution of uniformly discrete probability measures
$$\mu_{\{p_n\},\{\mathcal D_n\}}:=\delta_{p_1^{-1}\{0,a_1,b_1\}}\ast\delta_{(p_1p_2)^{-1}\{0,a_2,b_2\}}\ast\cdots $$
is a Moran measure with compact support if and only if
\begin{align*}
\sum_{n=1}^{\infty}|p_{1}p_{2}\cdots p_n|^{-1}d_n<\infty,\quad \mbox{where}\;d_n=\max\{0,|a_n|, |b_n|\}.
\end{align*}
Without the condition $\sup_{n\geq 1}\{\frac{|a_n|+|b_n|}{|p_n|}\}<\infty$, we give two sufficient conditions under which that $\mu_{\{p_n\},\{\mathcal D_n\}}$ is a spectral measure. If  $p_n=p>2$ and $\mathcal{D}_n=\{0,a_n,b_n\}$ with $\gcd(a_n,b_n)=1$, we also find an useful condition to guarantee that $\mu_{p,\{\mathcal D_n\}}$ is not a spectral measure. Our results extend some known theorems in An et al. [JFA, 2019] and Lu et al. [JFAA, 2022].
\end{abstract}

\medskip


\section {Introduction}
Let $\mu$ be a Borel probability measure on $\mathbb{R}^n$ and $\left\langle \cdot \right\rangle$ denote the standard inner product on $\mathbb R^n$. It is said that $\mu$ is a $spectral$ $measure$ if there exists a countable set $\Lambda\subset \mathbb{R}^n$ such that $E_\Lambda:=\{e^{-2\pi i\left\langle\lambda, x\right\rangle}:\lambda\in\Lambda\}$ forms an orthonormal basis for $L^2(\mu)$. In this case, the set $\Lambda$ is called a $spectrum$ of $\mu$. If the normalized Lebesgue measure restricting on a domain $\Omega\subset \mathbb{R}^n$ is a spectral measure, then we say $\Omega$ is a spectral set. The study of spectral measures originated from Fuglede conjecture \cite{Fuglede}.

{\bf  Fuglede conjecture}. A bounded measurable subset $\Omega\subset \mathbb{R}^n$ is a spectral set if and only if $\Omega$ is a tile.

Although the conjecture was disproved eventually \cite{Tao,FMM06, KM06}, it absorbs more attention on when it is true, see \cite{LM22}.
However, people also want to know whether there is a spectral measure that is not absolutely continuous with the Lebesgue measure. Fortunately, in 1998, Jorgensen and Pedersen \cite{JP98} discovered that some self-similar measures may also have spectra. A simple example is that the self-similar measure supported on quarter Cantor set. Following this surprising discovery, the spectrality and non-spectrality of various singular fractal measures, such as self-affine measures and Cantor-Moran measures, have been extensively studied, see \cite{AFL19,DHL14,DHL19,LW02,LMW22,Str00,Wu24} and references therein for details.
The aim of this paper is to study the spectrality and non-spectrality of a class Moran measures.
Moran measures are a fundamental class of fractal measures which is a probability measure defined on a Moran set, where mass is distributed iteratively according to a fixed set of "weights" \cite{F,Hut81}.

\begin{de}\label{de1.1}
Let $\{R_k\}_{k=1}^\infty$ be a sequence of $n\times n$ real expanding matrices, and let $\{\mathcal{D}_k\}_{k=1}^\infty$ be a sequence of digit sets in $\mathbb{R}^n$ satisfying
$2\leq \# \mathcal{D}_k<\infty$, where $\# \mathcal D_k$ is the cardinality of $\mathcal D_k$. Write $\mathcal{R}_j=R_j\cdots R_2R_1$ and $\delta_E=\frac1{\# E}\sum_{e\in E}\delta_e$, where $\delta_{e}$ is the Dirac measure at the point $e$. If  $\mathcal{U}_k:=\delta_{\mathcal{R}_1^{-1}\mathcal{D}_1}*\delta_{\mathcal{R}_2^{-1}\mathcal{D}_2}*\cdots*\delta_{\mathcal{R}_k^{-1}\mathcal{D}_k}$
converges weakly to a Borel probability measure, then we call this measure to be Moran measure and denoted it by
\begin{align}\label{(1.1)}
\mu_{\{ R_n\},\{\mathcal{D}_n\}}:=\delta_{\mathcal{R}_1^{-1}\mathcal{D}_1}*\delta_{\mathcal{R}_2^{-1}\mathcal{D}_2}*\cdots*\delta_{\mathcal{R}_k^{-1}\mathcal{D}_k}*\cdots.
\end{align}
\end{de}
The first Moran spectral measure was constructed by Strichartz \cite{Str00}. Subsequently, many Moran spectral measures have been found, see \cite{AH14,AFL19,AHH19,HH17,LLZ23,Shi19,Y22} and so on. In particular,  the spectrality of the following infinite convolution $\mu_{\{p_n\},\{\mathcal{D}_n\}}$ with three elements has been extensively studied in a series of paper \cite{AHH19,Ding,FW21,FDLW19,LDZ22,WYZ24,WDL18}.

Let $a_n$, $b_n$, $p_n\in\mathbb{Z}\setminus\{0\}$ with $|p_n|>1$ for $n\geq1$, and write
\begin{align*}
\mathcal{D}_n=\{0,a_n,b_n\},\;\;\; P_n=p_1\cdots p_n.
\end{align*}
It can be seen from \cite[Corollary 1.2]{LMW22}(see our Theorem \ref{thm2.3}) that the infinite convolution
$\mu_{\{p_n\},\{\mathcal D_n\}}$ is a Moran measure with compact support if and only if
\begin{align*}
\sum_{n=1}^{\infty}|P_n|^{-1}d_n<\infty,\quad \mbox{where}\;d_n=\max\{0,|a_n|, |b_n|\}.
\end{align*}
Many criteria for spectrality of $\mu_{\{p_n\},\{\mathcal{D}_n\}}$ have been established. In \cite{Ding}, Ding consider the case $\#\{a_n,b_n: n\geq1\}<\infty$ while Wang et al. \cite{WDL18} assume $\lim\sup_{n\to \infty}\frac{|a_n|+|b_n|}{|p_n|}<\infty$.  Fu et al. \cite{FDLW19} showed that $\mu_{\{p_n\},\{\mathcal D_n\}}$ is a spectral measure when  the sequence $\{\frac{|a_n|+|b_n|}{|p_n|}\}_{n=1}^\infty$ is bounded and $p_n\to\infty$ as $n$ tends infinity.
An et al. \cite{AHH19} proved that if $0<a_n<b_n$ and $\gcd(a_n,b_n)=1$ for all $n\geq1$, and $b_n<p_n$ for $n$ large enough, then $\mu_{\{p_n\},\{\mathcal D_n\}}$ is a spectral measure.
Lu et al. \cite{LDZ22} establish a necessary and sufficient condition for $\mu_{\{p_n\},\{\mathcal D_n\}}$ to be a spectral measure under the conditions $\gcd(a_n,b_n)=1$ and $\sup_{n\geq1}\{\frac{|a_n|}{|p_n|},\frac{|b_n|}{|p_n|}\}<\infty$.
As far as we know, all these results depend on the hypothesis $\sup_{n\geq 1}\{\frac{|a_n|+|b_n|}{|p_n|}\}<\infty$.
In this paper, contrast to these known results, there is no assumption $\sup_{n\geq 1}\{\frac{|a_n|+|b_n|}{|p_n|}\}<\infty$ in our setting.

The main results are as follows.

\begin{thm}\label{thm1.2}
Let $a_n,b_n\in\mathbb Z$ and $p_n\in 3\mathbb Z\setminus\{0\}$ satisfy $\{a_n,b_n\}=\{-1,1\}\pmod 3$. Assume that there exist large $N\in \mathbb N^{+}$ and constants $c_1 \geq c_2+1>2, c_{3}>1$ such that $|p_n|\geq c_{1}n$, $|a_n|\leq c_{2}n$ and $|b_n|\leq c_{3}n^k$ for $n>N$ and some $k\in\mathbb N^+$. Then $\mu_{\{p_n\},\{\mathcal D_n\}}$ is a spectral measure with a spectrum
\[
\Lambda = \left\{ \frac{1}{3} \sum_{k=1}^n d_k p_1 p_2 \cdots p_k: \text{ all } d_k \in \{-1, 0, 1\} \text{ and } n \geq 1 \right\}.
\]
\end{thm}

The proof of our Theorem \ref{thm1.2} is similar in spirit to the main results in \cite{AHH19,Ding,FW21,FDLW19,LDZ22,WYZ24,WDL18}, the key is to estimate the lower bound of the tail term of Fourier transform of $\mu_{\{p_n\},\{\mathcal D_n\}}$. However, Wang et al. \cite{WDL18}, Fu et al. \cite{FDLW19} and An et al. \cite{AHH19} used the condition $\sup_{n\geq 1}\{\frac{|a_n|+|b_n|}{|p_n|}\}<\infty$ to obtain that the tail term is greater than a positive constant. Through careful analysis, we find the fine structure of $p_n,\mathcal D_n$ rather than the boundedness condition can also make $\mu_{\{p_n\},\{\mathcal D_n\}}$ a spectral measure which is completely different from the existing spectral measures. Our viewpoint sheds some new light on the necessary and sufficient conditions for the spectrality of Moran measures with three-element digits.

Without the condition $\sup_{n\geq 1}\{\frac{|a_n|+|b_n|}{|p_n|}\}<\infty$, it is worth pointing out that if we write $\mu_{\{p_n\},\{\mathcal D_n\}}$ as an infinite convolution of some Dirac measures (see \eqref{b}), then the tail term $\nu_{>n_j}:=\delta_{p_{n_j+1}^{-1}\mathcal D_{n_j+1}}*\delta_{(p_{n_j+1} p_{n_j+2})^{-1}\mathcal D_{n_j+2}}*\cdots$ may not converges weakly. Lu et al. \cite{LDZ22} need that $\nu_{>n_j}$ has a weak limit $\nu$ to construct an admissible family or an equi-positive family. Although we can't construct an admissible family at present, we can still use some techniques to find an equi-positive family, which is the core strategy for proving the spectrality of the following $\mu_{\{p_n\},\{\mathcal D_n\}}$.
For the convenience,
suppose that $\{a_n,b_n\}\equiv\{a_n',b_n'\}\pmod {p_n}$, where $\{a_n',b_n'\}\subset\{0,1,\cdots ,p_n-1\}$.
Write $\alpha_n=|\frac{a_n-a_n'}{p_n}|$ and $\beta_n=|\frac{b_n-b_n'}{p_n}|$.

\begin{thm}\label{thm1.3}
Let $a_n,b_n\in\mathbb Z$ and $p_n\in 3\mathbb Z\setminus\{0\}$ satisfy $\{a_n,b_n\}=\{-1,1\}\pmod 3$. Assume $\lim_{n\rightarrow\infty}\frac{\max\{\alpha_n,\beta_n\}}{|p_{n-1}|}=0$. If $\liminf_{n\rightarrow\infty} |\frac{a_n}{p_n}|<\frac23$ or $\liminf_{n\rightarrow\infty} |\frac{b_n}{p_n}|<\frac23$, then $\mu_{\{p_n\},\{\mathcal D_n\}}$ is a spectral measure.
\end{thm}

Let $p_n=3n^2,\;a_n=\begin{cases}
3n^3+1, & n \text{ is odd,}\\
3n+1, & n \text{ is even.}\end{cases},\;b_n=3n^3+2$. It is easy to see that there are infinitely many $a_n, b_n$ bigger than $p_n$ and
 $\limsup_{n\rightarrow\infty}\frac{a_n}{p_n}=\limsup_{n\rightarrow\infty}\frac{b_n}{p_n}=\infty$.
 Hence, $\mu_{\{p_n\},\{\mathcal D_n\}}$ doesn't satisfy \cite[Theorem 1.3,Theorem 3.8,Theorem 5.1]{AHH19} and \cite[Theorem 1.2,Theorem 1.5]{LDZ22}.
 Despite this, we can verify that
 $$\lim_{n\rightarrow\infty}\frac{\max\{\alpha_n,\beta_n\}}{|p_{n-1}|}=0 \ {\it and}\  \liminf_{n\rightarrow\infty} |\frac{a_n}{p_n}|=0<\frac23.$$
 By Theorem \ref{thm1.3}, $\mu_{\{p_n\},\{\mathcal D_n\}}$ is a spectral measure.
Furthermore, if there exists a finite number of $n$ such that $\max\{a_n,b_n\}>p_n$, then $\lim_{n\rightarrow\infty}\frac{\max\{\alpha_n,\beta_n\}}{|p_{n-1}|}=0$. We are thus led to the  following Corollary.
\begin{cor}
Let $a_n,b_n\in\mathbb Z$ and $p_n\in 3\mathbb Z\setminus\{0\}$ satisfy $\{a_n,b_n\}=\{-1,1\}\pmod 3$. If there exists a finite number of $n$ such that $\max\{a_n,b_n\}>p_n$ and $\liminf_{n\rightarrow\infty} |\frac{a_n}{p_n}|<\frac23$ (or $\liminf_{n\rightarrow\infty} |\frac{b_n}{p_n}|<\frac23$), then $\mu_{\{p_n\},\{\mathcal D_n\}}$ is a spectral measure.
\end{cor}

For the non-spectrality of $\mu_{\{p_n\},\{\mathcal D_n\}}$, An et al. \cite{AHH19} obtained a complex sufficient condition. By restricting $p_n=p$, we can get another sufficient condition. However, our idea is different from \cite[Theorem 1.4]{AHH19} and the non-spectrality is the subtle part of this paper.

\begin{thm}\label{thm1.5}
Let $p,q\in \Bbb Z$ with $p>q\geq2$. Suppose that $p_n=p$ and $\gcd(a_n,b_n)=1$. If there exists a strictly monotone increasing sequence of positive integers $\{\omega_n\}_{n=1}^{\infty}$ and $c>0$ such that $c\cdot q^{\omega_n}\in\{a_{\omega_n},b_{\omega_n}\}$, then $\mu_{p,\{0,a_n,b_n\}}$  is not a spectral measure.
\end{thm}

The paper is organized as follows. Section \ref{section2} is devoted to introducing some preliminary results and proving Theorem \ref{thm1.2}. In Section \ref{section3}, we prove Theorem \ref{thm1.3} by constructing equi-positive family. Section \ref{section4} deals with the non-spectrality of $\mu_{\{p_n\},\{\mathcal D_n\}}$ and compare our Theorem \ref{thm1.5} with \cite[Theorem 1.4]{AHH19}.

\section{Preliminaries and proof of Theorem \ref{thm1.2}}\label{section2}

In this section, we give some necessary definitions and lemmas that are used in our proofs. In particular, we prove Theorem \ref{thm1.2}.

Recall that the Fourier transform of a Borel probability measure $\mu$ on $\mathbb{R}$ is defined by
$$
\hat{\mu}(\xi)=\int_{\mathbb{R}} e^{-2\pi i  \xi x}d\mu(x), \quad \xi\in\mathbb R.
$$
Let $a_n$, $b_n$, $p_n\in\mathbb{Z}\setminus\{0\}$ with $|p_n|>1$ for $n\geq1$, and $\mathcal{D}_n=\{0,a_n,b_n\}$.
Then we have
\begin{align}\label{q1}
\hat\mu_{\{p_n\},\{\mathcal D_n\}}(\xi)=\prod_{n=1}^{\infty}M_{\mathcal {D}_n}(P_n^{-1}\xi),
\end{align}
where $P_n=p_1\cdots p_n$ and the mask polynomial
\begin{align}\label{q2}
M_{\mathcal D_n}(\xi)=\frac13(1+e^{-2\pi i a_n \xi}+e^{-2\pi i b_n \xi}).
\end{align}
Denote by ${\mathcal Z}(f) := \{x\in\mathbb R : f(x) = 0\}$ the set of zeros of $f(x)$.
It is easy to see from \eqref{q1} and \eqref{q2} that
\begin{align}\label{q3}
\mathcal{Z}(M_{\mathcal{D}_n})=\frac1{3\gcd(a_n,b_n)}(\Bbb Z\setminus 3\Bbb Z) \quad \text{and}\quad \mathcal{Z}(\hat\mu_{\{p_n\},\{\mathcal D_n\}})=\bigcup_{n=1}^{\infty}\frac{P_n}{3\gcd(a_n,b_n)}(\Bbb Z\setminus 3\Bbb Z).
\end{align}

We say that $\Lambda$ is an orthogonal set of $\mu$ if $E_{\Lambda}=\{e^{-2\pi i\lambda x}:\lambda\in\Lambda\}$ is an orthonormal family for $L^2(\mu)$. It is easy to show that
$\Lambda$ is an orthogonal set for $\mu$ if and only if $\hat{\mu}(\lambda_i-\lambda_j)= 0$ for any two distinct $\lambda_i,\;\lambda_j\in\Lambda$, which is equivalent to
\begin{align}\label{q4}
(\Lambda-\Lambda)\setminus\{0\}\subseteq \mathcal{Z}(\hat{\mu}).
\end{align}
The difficulty in studying the spectrality of fractal measures lies in proving the completeness of the orthogonal set of exponential functions. In \cite{JP98}, Jorgensen and Pedersen provided a criterion that enables us to determine whether an orthogonal set is a spectrum of $\mu$.

\begin{lemma}\cite{JP98}\label{lem2.1}
Let $\mu$ be a probability measure in $\Bbb R^n$ with compact support, and let $\Lambda\subset \Bbb R^n$ be a countable subset. Define
\begin{align*}
Q_{\Lambda}(\xi)=\sum_{\lambda\in \Lambda}|\hat{\mu}(\xi+\lambda)|^2.
\end{align*}
Then the following conclusions hold.

$(i)$\;\;$E_\Lambda$  is an orthonormal set of $L^2(\mu)$ if and only if $Q_{\Lambda}(\xi)\leq 1$ for all $\xi\in\Bbb R^n$; in this case, $Q_{\Lambda}(\xi)$ is an entire function;

$(ii)\;\;\Lambda$ is a spectrum of $\mu$ if and only if $Q_{\Lambda}(\xi)\equiv 1$ for all $\xi\in\Bbb R^n$.
\end{lemma}

In this paper, we write
\begin{align}\label{w1}
\mu_n=\delta_{P_1^{-1}\mathcal D_1}*\cdots *\delta_{P_n^{-1}\mathcal D_n}\quad \text{and} \quad \mu_{>n}=\delta_{P_{n+1}^{-1}\mathcal D_{n+1}}*\delta_{P_{n+2}^{-1}\mathcal D_{n+2}}*\cdots.
\end{align}
Then $\mu_{\{p_n\},\{\mathcal D_n\}}=\mu_n*\mu_{>n}$. We define
\begin{align}\label{q5}
\Lambda_n=\sum_{i=1}^{n}\frac{P_i}3\{0,\pm 1\}\quad \text{and}\quad \Lambda=\bigcup_{n=1}^{\infty}\Lambda_n.
\end{align}
\begin{lemma}\label{lem2.2}
Let $a_n,b_n\in\Bbb Z$ and $p_n\in 3\Bbb Z\setminus\{0\}$ satisfy $\{a_n,b_n\}=\{-1,1\}\pmod 3$. Then the set $\Lambda_n$ is a spectrum of the measure $\mu_n$ and $\Lambda$ is an orthogonal set of the measure $\mu_{\{p_n\},\{\mathcal D_n\}}$.
\end{lemma}
\begin{proof}
It is easy to check that both $\Lambda_n$ and $\Lambda$ are orthogonal set by \eqref{q3} and \eqref{q4} for $\mu_n$ and $\mu_{\{p_n\},\{\mathcal{D}_n\}}$ respectively. Further, $\Lambda_n$ is a spectrum of $\mu_n$ because the dimension of $L^2(\mu_n)$ is $3^n$, which is the cardinality of the set $\Lambda_n$.
\end{proof}

The question of whether the infinite convolution $\mu_{\{p_n\},\{\mathcal{D}_n\}}$ defined by Definition \ref{de1.1} is a Moran measure has attracted widespread attention. Recently, Li et al. \cite{LMW22} investigated this issue and achieved some important results, one of which is the following.

\begin{thm}\cite[Theorem 3.26]{L20}\label{thm2.3}
With the hypothesis and notations in Definition \ref{de1.1}. Then $\mu_{\{R_n\},\{\mathcal{D}_n\}}$ is a Moran measure with compact support if and only if
\begin{align*}
\sum_{n=1}^{\infty}\max \{\|\mathcal{R}_n^{-1}d_n\|_{E}:d_n\in \mathcal D_n\}<\infty,
\end{align*}
where $\|\cdot \|_E$ denotes the Euclidean norm.
\end{thm}

We proceed to show the proof of Theorem \ref{thm1.2}.

\begin{proof}[Proof of Theorem \ref{thm1.2}]
Let $\Lambda_n$ and $\Lambda$ be defined by \eqref{q5}.  We divide the proof into two steps.

{\bf Step 1}: For $\lambda\in\Lambda_n$ and $\xi\in(-\frac13,\frac13)$, our primary goal is to find a constant $c>0$  such that for all $n\geq N$,
\begin{align}\label{a2.8}
|\hat{\mu}_{>n}(\xi+\lambda)|=\prod_{i=n+1}^{\infty}\frac13|1+e^{-2\pi i\frac{a_i}{P_i}(\xi+\lambda)}+e^{-2\pi i\frac{b_i}{P_i}(\xi+\lambda)}|\geq c>0.
\end{align}

For any $n\geq N$, $\lambda=\sum_{j=1}^n\frac{P_j}3\ell_j \in \Lambda_n$ with $\ell_j\in\{0,\pm1\}$ and $\xi\in(-\frac13,\frac13)$, since $|p_{n}|\geq c_{1}n$, we obtain
\begin{align*}
|\frac{\xi+\lambda}{P_n}|\leq\frac13(1+\frac1{|p_n|}+\frac1{|p_{n-1}p_n|}+\cdots+\frac1{|p_1\cdots p_n|})<\frac13\sum_{j=0}^{\infty}\frac1{c_{1}^j}=\frac{c_{1}}{3(c_{1}-1)}.
\end{align*}
Then for any $n\geq N$ and $j\geq 1$, since $|a_{n}|\leq c_{2}n$ and $c_1\geq c_2+1$, we have
\begin{align}\label{b1}
|\frac{a_{n+j}}{P_{n+j}}(\xi+\lambda)|\leq|\frac{c_{2} (n+j)}{p_{n+j}\cdots p_{n+1}}\frac{(\xi+\lambda)}{P_n}|< \frac{c_{2}}{ 3c_{1}^{j-1}(c_{1}-1)}\leq \frac{1}{3}.
\end{align}
By \cite[Lemma 4.9]{LDZ22}, i.e., by the continuity of $f(x, y) = 1 + e^{-2\pi ix} + e^{-2\pi iy}$, for $(x, y) \notin \pm \left( \frac{1}{3}, \frac{2}{3} \right) + \mathbb{Z}^2$, there exists a constant $\varepsilon_0>0$ such that
\begin{align}\label{d11}
\frac13|1+e^{-2\pi i\frac{a_{n+j}}{P_{n+j}}(\xi+\lambda)}+e^{-2\pi i\frac{b_{n+j}}{P_{n+j}}(\xi+\lambda)}|\geq\varepsilon_0.
\end{align}

We claim that there exists a constant $M_k\geq1$(depends only on $k$) such that $(n+j)^k\leq M_k(n+1)^j$ for all $j\geq k+1,\;n\in \Bbb N^+$.
Let $f(n,j)=\frac{(n+j)^k}{(n+1)^j}$ be defined on $(n,j)\in \Bbb N^+\times \{k+1,k+2,\cdots\}:=I$. Note that for all $n\geq2$ and $j\geq k+1$,
\begin{align}\label{a1}
\frac{f(n,j+1)}{f(n,j)}=(1+\frac1{n+j})^k\frac1{n+1}\leq (1+\frac1{k+3})^k\frac1{3}<\frac{e}{3}<1,
\end{align}
where the penultimate inequality holds because $(1+\frac1n)^n$ is monotonically increasing with respect to $n$ and approaches $e$. \eqref{a1} implies
\begin{align*}
\max_{(n,j)\in I}f(n,j)\leq \max_{(n,j)\in I}\{f(1,j),f(n,k+1)\}.
\end{align*}
It follows from $\lim_{j\rightarrow\infty}f(1,j)=\lim_{n\rightarrow\infty}f(n,k+1)=0$ that both $f(1,j)$ and $f(n,k+1)$ are bounded. Therefore, the claim follows.

For any $\lambda=\sum_{j=1}^n\frac{P_j}3\ell_j$ with $\ell_j\in\{0,\pm1\}$, $\xi\in(-\frac13,\frac13)$ and $j\geq k+1$, we have
\begin{align}\label{c1}
|\frac{b_{n+j}}{P_{n+j}}(\xi+\lambda)|
&\leq|\frac{c_{3}(n+j)^k}{p_{n+j}\cdots p_{n+1}}\frac{(\xi+\lambda)}{P_n}| \nonumber \\
&< \frac{c_{1}}{3(c_{1}-1)}\frac{c_{3}(n+j)^k}{c_{1}^{j}\prod_{i=1}^j(n+i)}
< \frac{c_{3}(n+j)^k}{3c_{1}^{j-1}(c_{1}-1)(n+1)^j} \nonumber\\
&\leq\frac{c_3M_{k}}{3c_{1}^{j-1}(c_{1}-1)}.
\end{align}
Let $\gamma=\max\{[\frac{\log(\frac{4c_{1}c_3M_{k}}{3(c_{1}-1)} )}{\log c_{1}}]+1,k+1,[\frac{\log {c_2}}{\log N}]+2\}$, where $[x]$ denotes the integer part of $x$. Then for any $j\geq \gamma$ and $n\geq N$, combining with \eqref{b1}, \eqref{c1} and $ c_3M_{k} \geq 1$, we obtain
\begin{align*}
&\max\{|\frac{a_{n+j}}{P_{n+j}}(\xi+\lambda)|,|\frac{b_{n+j}}{P_{n+j}}(\xi+\lambda)|\} \\
\leq & \max\{\frac{c_{2}}{ 3c_{1}^{j-1}(c_{1}-1)\prod_{i=1}^{j-1}(n+i)}, \frac{c_3M_{k}}{3c_{1}^{j-1}(c_{1}-1)} \} \\
=&\frac{c_3M_{k}}{3c_{1}^{j-1}(c_{1}-1)}<\frac{1}{4}.
\end{align*}
This implies that
\begin{align}\label{e1}
\{1,e^{-2\pi i\frac{a_{n+j}}{P_{n+j}}(\xi+\lambda)},e^{-2\pi i\frac{b_{n+j}}{P_{n+j}}(\xi+\lambda)}\}\subset\{x+yi\in\Bbb C: x^2+y^2= 1, x\geq \cos{\frac{2c_3M_{k}\pi}{3c_{1}^{j-1}(c_{1}-1)}}\}.
\end{align}
Hence, for any $n\geq N$, $\lambda\in\Lambda_n$ and $\xi\in(-\frac13,\frac13)$, by \eqref{d11} and \eqref{e1}, we obtain
\begin{align*}
|\hat{\mu}_{>n}(\xi+\lambda)|&\geq \varepsilon_0^{\gamma-1}\prod_{j=\gamma}^{\infty}\frac13|1+e^{-2\pi i\frac{a_{n+j}}{P_{n+j}}(\xi+\lambda)}+e^{-2\pi i\frac{b_{n+j}}{P_{n+j}}(\xi+\lambda)}|\nonumber\\
&\geq \varepsilon_0^{\gamma-1}\prod_{j=\gamma}^{\infty}\cos{\frac{2c_3M_{k}\pi}{3c_{1}^{j-1}(c_{1}-1)}}:=c.
\end{align*}
{\bf Step 2}: We prove $\Lambda$ is a spectrum of $\mu_{\{p_n\},\{\mathcal D_n\}}$. Let $\alpha_n=N+n$. By Lemma \ref{lem2.2}, we obtain $\Lambda_{\alpha_n}$ is a spectrum of $\mu_{\alpha_n}$. It follows from \eqref{a2.8} that $$\inf_{\lambda\in\Lambda_{\alpha_{n+s}\setminus \Lambda_{\alpha_n}}}|\hat{\mu}_{>\alpha_{n+s}}(\xi+\lambda)|^2\geq c^2>0$$
for $n,s\geq 1$. Combining with \cite[Theorem 2.3 (i)]{AHH19}, we know $\mu_{\{p_n\},\{\mathcal D_n\}}$ is a spectral measure with a spectrum $\Lambda$ in \eqref{q5}.
\end{proof}

\section{Proof of Theorem \ref{thm1.3}}\label{section3}
The purpose of this section is to complete the proof of Theorem \ref{thm1.3}. Recall that $\mu_{\{p_n\},\{\mathcal D_n\}}=\mu_n*\mu_{>n}$. We define
\begin{align}\label{b}
\nu_{>n}=\delta_{p_{n+1}^{-1}\mathcal D_{n+1}}*\delta_{(p_{n+1} p_{n+2})^{-1}\mathcal D_{n+2}}*\cdots.
\end{align}
Then $\nu_{>n}(\cdot)=\mu_{>n}(\frac1{p_1p_2\cdots p_n}\cdot )$. Clearly, the spectrality of $\mu_{\{p_n\},\{\mathcal D_n\}}$ is affected by the tail term $\nu_{>n}$. To study the properties of $\nu_{>n}$, the equi-positivity condition is a very useful tool. An et al. \cite{AFL19} introduced equi-positivity of a family of probability measures supported on the common compact subset. Later, Li et al. \cite{LMW24} proposed a more generalized equal positivity condition (The support of a family of probability measures does not need to be contained in the common compact subset).

\begin{de}\cite[Definition 4.1]{LMW24}
We call $\Phi\subset \mathcal P(\Bbb R)$ an equi-positive family if there exist $\varepsilon>0$ and $\delta>0$ such that for $x\in[0,1)$ and $\mu\in\Phi$, there exists an integral vector $k_{x,\mu}\in \Bbb Z$ such that
\begin{align*}
|\hat{\mu}(x+y+k_{x,\mu})|\geq \varepsilon,
\end{align*}
for all $|y|<\delta$, where $k_{x,\mu}=0$ for $x=0$.
\end{de}

To prove Theorem \ref{thm1.3}, we first present an essential lemma proved by Li et al.

\begin{lemma}\cite[Theorem 4.2]{LMW24}\label{lem3.2}
Let $\{(p_n,\mathcal D_n,L_n):n\geq 1\}$ be a sequence of Hadamard triples in $\mathbb{R}$. Let the probability measure $\mu$ be the weak limit of $\mu_n$ defined by \eqref{w1}. If there exists a subsequence $\{n_j\}$ of positive integers such that the family $\{\nu_{>n_j}\}$ defined by \eqref{b} is an equi-positive family, then $\mu$ is a spectral measure with a spectrum in $\Bbb Z$.
\end{lemma}

\begin{proof}[Proof of Theorem \ref{thm1.3}]
We choose $L_n=\{0,\frac13,\frac23\}p_n$. Since  $p_n\in 3\Bbb Z\setminus\{0\}$ and $\{a_n,b_n\}=\{-1,1\}\pmod 3$, it is easy to check $(p_n,\mathcal D_n,L_n)$ is a Hadamard triple for all $n$. It follows from  $\{a_n',b_n'\}\subset\{0,1,\cdots, p_n-1\}$ and $\lim_{n\rightarrow\infty}\frac{\max\{\alpha_n,\beta_n\}}{|p_{n-1}|}=0$ that there exists a constant $c$ satisfying
\begin{align*}
\sum_{n=1}^{\infty}\frac{\max\{|a_n|,|b_n|\}}{|p_1p_2\cdots p_{n}|}\asymp \sum_{n=2}^{\infty}\frac{\max\{\alpha_n,\beta_n\}}{|p_1p_2\cdots p_{n-1}|}
\leq \sum_{n=1}^{\infty}\frac{c}{3^j}<\infty,
\end{align*}
where $A\asymp B$ means there exist constants $c_{1},c_{2}$ such that $c_{1}B\leq A \leq c_{2}B$.
This ensures there exists a probability measure $\mu$ such that $\mu_n$ converges weakly to $\mu$.

Without loss of generality, we assume $\liminf_{n\rightarrow\infty} |\frac{a_n}{p_n}|<\frac23$. Then there exists a strictly increasing sequence $\{n_j\}_{j=1}^{\infty}$ of positive integers such that $\lim_{j\rightarrow\infty} |\frac{a_{n_j}}{p_{n_j}}|=r<\frac23$. We will prove that $\{\nu_{>n_j-1}\}_{j=N}^{\infty}$ is an equi-positive family for some large integer $N$. Then Theorem \ref{thm1.3} follows from Lemma \ref{lem3.2}.

Let $\delta=\min\{\frac2{2+3r}-\frac12,\frac14\}>0$. In order to prove $\{\nu_{>n_j-1}\}_{j=N}^{\infty}$ is an equi-positive family, we only need to prove the following claim.

\textbf {Claim}: There exist two constant $C\in(0,1)$ and $N\in \Bbb N$ such that $|\hat \nu_{>n_k-1}(x)|>C$ for all $x\in[-\frac12-\delta,\frac12+\delta]$ and $k\geq N$.

In fact, If the claim holds, we put $k_x=0$ for $x\in[0,\frac12)$ and $k_x=-1$ for $x\in[\frac12,1)$. Hence $|\hat \nu_{>n_k-1}(x+y+k_x)|>C$ for all $x\in[0,1)$ and $|y|<\delta$. This implies $\{\nu_{>n_k-1}\}_{k=N}^{\infty}$ is an equi-positive family. Below is the proof of the claim.

It is easily seen that
\begin{align}\label{d1}
\hat \nu_{>n_k-1}(x)=M_{\mathcal D_{n_k}}(x_0)M_{\mathcal D_{n_k+1}}(x_1)\prod_{j=2}^{\infty}M_{\mathcal D_{n_k+j}}(x_j),
\end{align}
where $x_j=\frac{x}{p_{n_k}p_{n_k+1}\cdots p_{n_k+j}}$ for all $j\geq 0$.
For all $x\in[-\frac12-\delta,\frac12+\delta]$, we have
\begin{align}\label{aa}
|x_j|\leq \frac1{4\cdot 3^{j-1}|p_{n_k+j}|}.
\end{align}
Then we estimate the lower bounds of $|M_{\mathcal D_{n_k}}(x_0)|$, $|M_{\mathcal D_{n_k+1}}(x_1)|$ and $\prod_{j=2}^{\infty}|M_{\mathcal D_{n_k+j}}(x_j)|$, respectively.

Firstly, we estimate $|M_{\mathcal D_{n_k}}(x_0)|$.
It is evident that there exists a large integer $N_1$ such that $|\frac{a_{n_k}}{p_{n_k}}|<\frac13+\frac r2$ for all $k\geq N_1$. Then $|\frac{a_{n_k}}{p_{n_k}}x|<(\frac13+\frac r2)(\frac12+\delta)\leq\frac13$. By \cite[Lemma 4.9]{LDZ22}, there exists a constant $C_{r,\delta}\in(0,1)$ such that
\begin{align}\label{d2}
|M_{\mathcal D_{n_k}}(x_0)|=\frac13|1+e^{-2\pi i \frac{a_{n_k}}{p_{n_k}}x}+e^{-2\pi i \frac{b_{n_k}}{p_{n_k}}x}|\geq C_{r,\delta} \quad \text{for all }k\geq N_1.
\end{align}

Next, we estimate $|M_{\mathcal D_{n_k+1}}(x_1)|$. Note that
\begin{align}\label{bb1}
|M_{\mathcal D_{n_k+1}}(x_1)|&\geq\frac13|1+e^{-2\pi i a_{n_k+1}'x_1}+e^{-2\pi i b_{n_k+1}'x_1}|-\frac13|e^{-2\pi i a_{n_k+1}'x_1}-e^{-2\pi i a_{n_k+1}x_1}\nonumber\\
&+e^{-2\pi i b_{n_k+1}'x_1}-e^{-2\pi i b_{n_k+1}x_1}|:=I_{k,1}-I_{k,2}.
\end{align}
Clearly, $a_{n_k+1}'\in\{0,1,\cdots,p_{n_k+1}-1\}$. Combining this with \eqref{aa}, we have
\begin{align*}
|a_{n_k+1}'x_1|=|a_{n_k+1}'\frac{x}{p_{n_k}p_{n_k+1}}|<|\frac x{p_{n_k}}|\leq \frac14.
\end{align*}
Again by \cite[Lemma 4.9]{LDZ22}, there exists a constant $C_0\in(0,1)$ such that for all $k$,
\begin{align}\label{bb2}
I_{k,1}=\frac13|1+e^{-2\pi i a_{n_k+1}'x_1}+e^{-2\pi i b_{n_k+1}'x_1}|>C_0.
\end{align}
Observe that
\begin{align*}
I_{k,2}&=\frac13|e^{-2\pi i a_{n_k+1}'x_1}-e^{-2\pi i a_{n_k+1}x_1}+e^{-2\pi i b_{n_k+1}'x_1}-e^{-2\pi i b_{n_k+1}x_1}|\nonumber\\
&\leq\frac13|e^{-2\pi i a_{n_k+1}'x_1}-e^{-2\pi i a_{n_k+1}x_1}|+\frac13|e^{-2\pi i b_{n_k+1}'x_1}-e^{-2\pi i b_{n_k+1}x_1}|\nonumber\\
&=\frac23(|\sin((a_{n_k+1}-a'_{n_k+1})x_1\pi)|+|\sin((b_{n_k+1}-b'_{n_k+1})x_1\pi)|).
\end{align*}
Since $|\sin x|\leq |x|$, $x_1=\frac x{p_{n_k}p_{n_k+1}}$ and $|x|\leq \frac12+\delta\leq \frac34$, we have
\begin{align*}
I_{k,2}\leq \frac{\pi}2(\frac{\alpha_{n_k+1}}{|p_{n_k}|}+\frac{\beta_{n_k+1}}{|p_{n_k}|}).
\end{align*}
It follows from $\lim_{n\rightarrow\infty}\frac{\max\{\alpha_n,\beta_n\}}{|p_{n-1}|}=0$ that there exists a large integer $N_2$ such that for all $k\geq N_2$,

\begin{align}\label{bb3}
I_{k,2}<\frac{C_0}2.
\end{align}
Combining \eqref{bb1}, \eqref{bb2} and \eqref{bb3}, we have
\begin{align}\label{d3}
|M_{\mathcal D_{n_k+1}}(x_1)|\geq \frac{C_0}2\quad \text{for all }k\geq N_2.
\end{align}

Finally, we need accurate estimate for $\prod_{j=2}^{\infty}|M_{\mathcal D_{n_k+j}}(x_j)|$. Obviously,
\begin{align}\label{qw4}
|M_{\mathcal D_{n_k+j}}(x_j)|&\geq\frac13|1+e^{-2\pi i a_{n_k+j}'x_j}+e^{-2\pi i b_{n_k+j}'x_j}|-\frac13|e^{-2\pi i a_{n_k+j}'x_j}-e^{-2\pi i a_{n_k+j}x_j}\nonumber\\
&+e^{-2\pi i b_{n_k+j}'x_j}-e^{-2\pi i b_{n_k+j}x_j}|:=I_{k,j}'-I_{k,j}''.
\end{align}
From \eqref{aa}, it follows that for each $j\geq 2$ and $k\in \Bbb N$,
\begin{align*}
\max\{|2\pi a_{n_k+j}'x_j|,|2\pi b_{n_k+j}'x_j|\}\leq\frac{2\pi (|p_{n_k+j}|-1)}{4\cdot 3^{j-1}|p_{n_k+j}|}<\frac{\pi}{2\cdot 3^{j-1}}<\frac{\pi}2.
\end{align*}
This implies that
\begin{align*}
\{1,e^{-2\pi ia_{n_k+j}'x_j},e^{-2\pi i b_{n_k+j}'x_j}\}\subset\{x+yi\in\Bbb C: x^2+y^2= 1, x\geq \cos{\frac{\pi}{2\cdot 3^{j-1}}}\}.
\end{align*}
Then we have
\begin{align}\label{qw2}
I_{k,j}'=\frac13|1+e^{-2\pi i a_{n_k+j}'x_j}+e^{-2\pi i b_{n_k+j}'x_j}|\geq \cos{\frac{\pi}{2\cdot 3^{j-1}}}.
\end{align}
It is easy to see that
\begin{align*}
I_{k,j}''&=\frac13|e^{-2\pi i a_{n_k+j}'x_j}-e^{-2\pi i a_{n_k+j}x_j}+e^{-2\pi i b_{n_k+j}'x_j}-e^{-2\pi i b_{n_k+j}x_j}|\nonumber\\
&\leq\frac13|e^{-2\pi i a_{n_k+j}'x_j}-e^{-2\pi i a_{n_k+j}x_j}|+\frac13|e^{-2\pi i b_{n_k+j}'x_j}-e^{-2\pi i b_{n_k+j}x_j}|\nonumber\\
&=\frac23(|\sin{\frac{\alpha_{n_k+j} x\pi}{p_{n_k}p_{n_k+1}\cdots p_{n_k+j-1}}}|+|\sin{\frac{\beta_{n_k+j} x\pi}{p_{n_k}p_{n_k+1}\cdots p_{n_k+j-1}}}|).
\end{align*}
By using $|\sin x|\leq x$ and $|x|\leq \frac12+\delta\leq \frac34$, we obtain
\begin{align*}
I_{k,j}''\leq\frac{\pi}2(|\frac{\alpha_{n_k+j}}{p_{n_k}p_{n_k+1}\cdots p_{n_k+j-1}}|+|\frac{\beta_{n_k+j}}{p_{n_k}p_{n_k+1}\cdots p_{n_k+j-1}}|).
\end{align*}
Since $\lim_{n\rightarrow\infty}\frac{\max\{\alpha_n,\beta_n\}}{|p_{n-1}|}=0$, there exists a large integer $N_3$ such that $\frac{\max\{\alpha_{n_k+j},\beta_{n_k+j}\}}{|p_{n_k+j-1}|}<\frac1{3\pi}$ for all $k\geq N_3$ and $j\geq 2$. Then we have
\begin{align}\label{qw3}
I_{k,j}''\leq\frac1{3^{j-1}}.
\end{align}
Combining \eqref{qw4}, \eqref{qw2} and \eqref{qw3},
\begin{align*}
|M_{\mathcal D_{n_k+j}}(x_j)|\geq \cos{\frac{\pi}{2\cdot 3^{j-1}}}-\frac{1}{3^{j-1}}>0.
\end{align*}
Note that
\begin{align*}
\cos{\frac{\pi}{2\cdot 3^{j-1}}}-\frac{1}{3^{j-1}}-1\leq 0 \quad \text{and}\quad \lim_{j\rightarrow\infty}\cos{\frac{\pi}{2\cdot 3^{j-1}}}-\frac{1}{3^{j-1}}-1=0.
\end{align*}
By using $\log (x+1)\geq \frac32 x\;\text{for }x\in(-\frac12,0]$ and $\cos x\geq 1-\frac12 x^2$, we obtain that for large integer $j$,
\begin{align*}
\log(\cos{\frac{\pi}{2\cdot 3^{j-1}}}-\frac{1}{3^{j-1}})>\frac32(\cos{\frac{\pi}{2\cdot 3^{j-1}}}-\frac{1}{3^{j-1}}-1)\geq-\frac32(\frac{\pi^2}{8\cdot 9^{j-1}}+\frac1{3^{j-1}}).
\end{align*}
Then there exists a constant $C_1>0$ such that
\begin{align*}
\sum_{j=2}^{\infty}\log(\cos{\frac{\pi}{2\cdot 3^{j-1}}}-\frac{1}{3^{j-1}})>-C_1.
\end{align*}
Hence, we get
\begin{align}\label{d4}
\prod_{j=2}^{\infty}|M_{\mathcal D_{n_k+j}}(x_j)|&=\exp\{\sum_{j=2}^{\infty}\log|M_{\mathcal D_{n_k+j}}(x_j)|\}\geq\exp\{\sum_{j=2}^{\infty}\log(\cos{\frac{\pi}{2\cdot 3^{j-1}}}-\frac{1}{3^{j-1}})\}\nonumber\\
&>e^{-C_1}\quad \text{for all } k\geq N_3.
\end{align}

We choose
\begin{align*}
N=\max_{1\leq i\leq 3}\{N_i\}\quad \text{and} \quad C=\frac{C_0C_{r,\delta}e^{-C_1}}2\in(0,1).
\end{align*}
Then for all $x\in[-\frac12-\delta,\frac12+\delta]$ and $k\geq N$, combining \eqref{d1}, \eqref{d2}, \eqref{d3} and \eqref{d4}, we have
 \begin{align*}
|\hat \nu_{>n_k-1}(x)|>C.
\end{align*}
Therefore, the claim holds.
\end{proof}

\section {Non-spectrality of $\mu_{\{p_n\},\{\mathcal{D}_n\}}$}\label{section4}

In this section, we aim to prove Theorem \ref{thm1.5}. This proof is both challenging and meticulous, forming the core of our paper. To do this, we will introduce the maximal mapping, which is a typical way of constructing the maximal orthogonal set. Write
\begin{align*}
A=\{0,1,2\},\quad  A^0=\{\emptyset\}\quad\mbox{and}\quad A^n=\{\textbf i=i_1i_2\cdots i_n:\; \mbox{all} \; i_j\in A\}.
\end{align*}
Let $A^*=\cup_{n=1}^{\infty}A^n$ be the set of all finite words and $A^{\infty}=\{\textbf{i}=i_1i_2\cdots:\;\mbox{all} \; i_j\in A\}$ be the set of all infinite words. For any $\textbf{i}\in A^*$, $\textbf{j}\in A^*\cup A^{\infty}$, \textbf{ij} is their natural conjunction. In particular, $\emptyset\textbf i=\textbf i$, $\textbf i0^{\infty}=\textbf i00\cdots$ and $0^k=0\cdots 0\in A^k$.

\begin{de}
Suppose that $3|p_n$ for $n\geq2$. A mapping $\iota: A^*\rightarrow \Bbb Z $ is called a maximal mapping if

(i) $\iota(\emptyset)=\iota(0^n)=0$ for all $n\geq 1$;

(ii) for any $\mathbf{i}=i_1i_2\cdots i_k\in A^k$, $\iota(\mathbf{i})\in (i_k+3\Bbb Z)\cap \{-1,0,1,\cdots, |p_{k+1}|-2\}$;

$(iii)$ for any $\mathbf{i}\in A^*$, there exists $\mathbf{j}\in A^*$ such that $\iota$ vanishes eventually on $\mathbf{ij}0^{\infty}$, i.e., $\iota(\mathbf{ij}0^k)=0$ for sufficient large $k$.
\end{de}

Let
\begin{align*}
A^{\iota}=\{\textbf{i}=i_1\cdots i_k\in A^*:i_k\neq 0, \; \iota(\textbf{i}0^n)=0\; \mbox{for sufficient large}\;n\}\cup \{\emptyset\}
\end{align*}
and define
\begin{align*}
\iota^*(\textbf{i})=\sum_{k=1}^{\infty}\iota(\textbf{i}0^\infty|_k)|P_k|,\quad \textbf i \in A^{\iota},
\end{align*}
where $\textbf{i}0^{\infty}|_n$ denotes the word of the first $n$ entries.

\begin{thm}\cite[Theorem 3.4]{AHH19}\label{thm4.2}
Let $\mu_{\{p_n\},\{\mathcal{D}_n\}}$ be defined by \eqref{(1.1)}. Suppose that $\gcd(a_n,b_n)=1$, $\{a_n,b_n\}\equiv\{1,2\}\pmod3$ and $a_n<b_n$ for $n\geq1$. If $3|p_n$ for $n\geq2$, then $\Lambda$ is a maximal orthogonal set of $\mu_{\{p_n\},\{\mathcal D_n\}}$ if and only if there exists a maximal mapping $\iota$ such that $\Lambda=\frac13\iota^*(A^{\iota})$.
\end{thm}

Let $\mu_{n}$ be given in \eqref{w1}. An et al. \cite{AHH19} established the following result by using the maximal mapping and Theorem \ref{lem2.1}.

\begin{lemma}\cite[Lemma 3.5]{AHH19}\label{lem4.3}
With the hypothesis of Theorem \ref{thm4.2}.  Then $A^{\iota}\cap A^n$ is an orthogonal set of $\mu_{n}$. Consequently,
\begin{align*}
\sum_{\mathbf{i}\in A^{\iota}\cap A^n}|\hat\mu_n(\xi+\frac 1 3 \iota^*(\mathbf{i}))|^2\leq 1,\quad \forall\;\xi\in\mathbb{R}.
\end{align*}
\end{lemma}

We give a useful lemma, which is used to prove Theorem \ref{thm1.5}.

\begin{lemma}\label{lem4.4}
Let $\alpha$ and $\beta$ be two co-prime positive integers such that $\alpha>\beta>1$. Then there exists a sequence $\{q_n\}_{n=1}^\infty\subset\mathbb{N}$ such that
\begin{align*}
\Delta(q_n+1,x):=\#\left\{j:0\leq j\leq q_n+1 \;\mathrm{and} \;||(\frac\alpha \beta)^{j}x||\geq\frac1{\alpha\beta}\right\}\geq n
\end{align*}
for all $x\in[1,\frac\alpha\beta]$, where $\|x\|=dist(x,\Bbb Z)$. In particular, $q_1$ is the smallest positive integer such that $[\frac\alpha\beta]+1\leq \beta^{q_1}$.
\end{lemma}

\begin{proof}
For $n\geq1$ and integers $\alpha>\beta>1$, write
\begin{align*}
B(n)=\left(n-\frac{1}{\alpha\beta},n+\frac{1}{\alpha\beta}\right)\quad \mbox{and} \quad C=\bigcup_{n=1}^\infty \left[n+\frac 1 {\alpha\beta}, n+1-\frac 1 {\alpha\beta}\right].
\end{align*}
Let $q_1$ be the smallest positive integer such that $[\frac\alpha\beta]+1\leq \beta^{q_1}$.
For all $x\in[1,\frac\alpha\beta]$, we firstly prove
\begin{align}\label{q6}
\Delta(q_1+1,x):=\#\left\{j:0\leq j\leq q_1+1 \;\mathrm{and} \;||(\frac\alpha \beta)^{j}x||\geq\frac1{\alpha\beta}\right\}\geq 1.
\end{align}
It is easy to check that $x\in[1,\frac\alpha\beta]\subset C \cup (\cup_{i=1}^{[\frac{\alpha}{\beta}]+1} B(i))$. Obviously, \eqref{q6} holds for $x\in C$.
If $x\in\cup_{i=1}^{[\frac{\alpha}{\beta}]+1} B(i)$, then there exists $i_0\in\{1,2,\cdots ,[\frac{\alpha}{\beta}]+1\}$ such that $x\in B(i_0)$, i.e.,
$\frac\alpha\beta x\in(\frac\alpha\beta i_0-\frac1{\beta^2},\frac\alpha\beta i_0+\frac1{\beta^2})$. We can prove the \eqref{q6} in two cases.

{\bf Case I}: $\beta \nmid i_0$. Since $\gcd(\alpha,\beta)=1$, we put $\alpha i_0=k \beta+l$, where $k\in\Bbb N^+$ and $l\in\{1,2,\cdots, \beta-1\}$. Then
\begin{align*}
\frac\alpha\beta x\in(k+\frac{l}\beta -\frac1{\beta^2},k+\frac {l}\beta +\frac1{\beta^2}).
\end{align*}
By a simple calculation, we get $\frac\alpha\beta x\in[k+\frac1{\alpha\beta},k+1-\frac1{\alpha\beta}] \subset C$.

{\bf Case II}: $\beta | i_0$. Write $i_0=\beta^s \gamma$, where $1\leq s\leq q_1$ and $\beta\nmid\gamma $. If $(\frac\alpha\beta)^{j}x\in C$ for some $1\leq j\leq s$, then the claim follows. Otherwise, we have
\begin{align*}
\frac\alpha\beta x\in B(\alpha \beta^{s-1}\gamma), \; \cdots, \;(\frac\alpha\beta)^{s}x\in B(\alpha^s \gamma).
\end{align*}
Since $\beta\nmid\alpha^s\gamma$, similar to Case I above, we also obtain $(\frac\alpha\beta)^{s+1}x\in C$.

Based on the above, we complete the proof of \eqref{q6}.

Clearly, for all $x\in[1,\frac\alpha\beta]$, there exists a constant $q_2>q_1+2$ such that
\begin{align*}
(\frac\alpha\beta)^{q_1+2}x \in C \cup (\cup_{j=1}^{\beta^{q_2-q_1-2}}B(j)).
\end{align*}
Hence, $\Delta(q_2+1,x)\geq 2$ is equivalent to
\begin{align}\label{(8.1)}
\#\left\{j:q_1+2\leq j\leq q_2+1\;\mathrm{and}\;||(\frac\alpha\beta)^jx||\geq \frac1{\alpha\beta}\right\}\geq1.
\end{align}
The proof of \eqref{(8.1)} is divided into the following three cases.

{\bf Case I}: $(\frac\alpha\beta)^{q_1+2}x\in C$. Then \eqref{(8.1)} holds.

{\bf Case II}: $(\frac\alpha\beta)^{q_1+2}x\in B(i)$ with $\beta\nmid i$. Analogous to Case I in our discussion of \eqref{q6}, we have $(\frac\alpha\beta)^{q_1+3}x\in C$.
This implies that $\eqref{(8.1)}$ holds.

{\bf Case III}: $(\frac\alpha\beta)^{q_1+2}x\in B(i)$ with $\beta|i$ and $i\leq \beta^{q_2-q_1-2}$. Similar to Case II in our discussion of \eqref{q6}, we also obtain \eqref{(8.1)}.

Summarizing the above three cases, we know that $\Delta(q_2+1,x)\geq 2$.

Note that
\begin{align*}
(\frac\alpha\beta)^{q_2+2}x \in C \cup (\cup_{j=1}^{\beta^{q_3-q_2-2}}B(j))
\end{align*}
for some integer $q_3>q_2+2$.  Arguing similarly for \eqref{(8.1)}, we find that
\begin{align*}
\#\Big\{j:q_2+2\leq j\leq q_3+1\;\mathrm{and}\;||(\frac\alpha\beta)^jx||\geq\frac1{\alpha\beta}\Big\}\geq1
\end{align*}
for all $x\in[1,\frac\alpha\beta]$. Hence   $\Delta(q_3+1,x)\geq 3$ for all $x\in[1,\frac\alpha\beta]$.

When $n\geq4$, we see from the above argument that there exists an integer $q_n$ such that $\Delta(q_n+1,x)\geq n$ for all $x\in[1,\frac\alpha\beta]$.
Summarizing, the proof of this lemma is complete.
\end{proof}

We are now in a position to prove Theorem \ref{thm1.5}.

\begin{proof}[Proof of Theorem \ref{thm1.5}]
According to \cite[Theorem 1.2]{AHH19}, we only need to consider the case $p\in 3\Bbb N$ and $\{a_n,b_n\}=\{-1,1\}\pmod 3$ for $n\geq1$. Without loss of generality, we assume that $\omega_n=\omega_1+(n-1)d$ with $d\in \mathbb N^{+}$ and $b_{\omega_n}=c\cdot q^{\omega_n}$.

{\bf Step 1. Constructing maximal orthogonal set by  maximal mapping.}
Let $\Lambda$ be a maximal orthogonal set of $\mu_{\{p_n\},\{D_n\}}$.  It follows from Theorem \ref{thm4.2} that there exists a maximal mapping $\iota$ such that $\Lambda=\frac13 \iota^*(A^{\iota})$. For positive integers $n$ and $k$, write
\begin{align*}
Q_{n,k}(\xi)=\sum_{\textbf{i}\in I^k_n}|\hat\mu_{\{p_n\},\{\mathcal{D}_n\}}(\xi+\frac13\iota^*(\textbf i))|^2 \quad\mbox{and}\quad Q(\xi)=\sum_{\textbf{i}\in A^{\iota}}|\hat\mu_{\{p_n\},\{\mathcal{D}_n\}}(\xi+\frac13\iota^*(\textbf i))|^2,
\end{align*}
where $I^k_n=\{\textbf{i}\in A^{\iota}:|\textbf{i}|\leq (n d)^k\}$.

{\bf Step 2. Estimate the tail term. }
Let $\mu_{>n}$ be given in \eqref{w1}. We write $\frac{q^d}{p^d}=\frac\beta\alpha$, where $\gcd(\alpha,\beta)=1$. It is easy to check that
\begin{align}\label{(8.3)}
|\hat\mu_{>((n+1)d)^k}(\xi+\frac 1 3 \iota^*(\textbf{i}))|&\leq  \prod_{j=1}^{\infty}\frac 1 3 (1+|1+\exp\{{-2\pi i\frac{b_{((n+1)d)^k+jd+\omega_1}}{P_{((n+1)d)^k+jd+\omega_1}}}(\xi+\frac13\iota^*(\mathbf{i}))\}|)\nonumber\\
&= \prod_{j=1}^{\infty}\frac 1 3 (1+|1+\exp\{-2\pi i(\frac {q}{p})^{jd}\frac{cq^{((n+1)d)^k+\omega_1}(\xi+\frac13\iota^*(\mathbf{i}))}{p^{((n+1)d)^k+\omega_1}}\}|)\nonumber\\
&:=\prod_{j=1}^{\infty}\frac 1 3 (1+|1+\exp\{-2\pi i (\frac\beta\alpha)^j {y(n,k,\xi,\textbf{i})}\}|).
\end{align}
Define $ I^k_{n,n+1}=\{\textbf{i}\in A^{\iota}:(nd)^k<|\textbf{i}|\leq ((n+1)d)^k\}$. For any $\mathbf{i}\in I^k_{n,n+1}$, let $\ell$ be the largest index such that $\iota(\textbf{i}0^{\infty}|_{\ell})\neq 0$. It is easy to see that $\ell\geq (nd)^k+1$ and
\begin{align*}
|\iota^*(\textbf{i})|=\left|\sum_{s=1}^{\infty}\iota(\textbf{i}0^\infty|_s)|P_s|\right|\geq |P_{\ell}|-\sum_{k=1}^{\ell-1}(|p_{k+1}|-2)|P_k|\geq |P_{(nd)^k}|=p^{(nd)^k}.
\end{align*}
Thus, for any $\xi\in[0,\frac 1 3]$, $k\in\Bbb N$, large $n$ and $\textbf{i}\in I^k_{n,n+1}$, we have
\begin{align}\label{(8.4)}
|y(n,k,\xi,\textbf{i})|=\left|\frac{cq^{((n+1)d)^k+\omega_1}(\xi+\frac13\iota^*(\mathbf{i}))}{p^{((n+1)d)^k+\omega_1}}\right|\geq\frac{cq^{((n+1)d)^k+\omega_1}}{6p^{((n+1)d)^k-(nd)^k+\omega_1}}\geq q^{n^k}.
\end{align}

Choose a positive integer $r(n,k,\xi,\textbf{i})$ satisfying $(\frac\beta\alpha)^{r(n,k,\xi,\textbf{i})}|y(n,k,\xi,\textbf{i})|\in(1,\frac\alpha\beta]$.
According to Lemma \ref{lem4.4}, there exists a strictly increasing sequence $\{j_i\}_{i=1}^{s(n,k,\xi,\textbf{i})}\subset\{0,1,\cdots,r(n,k,\xi,\textbf{i})\}$ such that
\begin{align*}
\|(\frac \beta\alpha)^{j_i} y(n,k,\xi,\textbf{i})\|\geq \frac1{\alpha\beta}\quad \mbox{and}\quad |(\frac\beta\alpha)^{j_i} y(n,k,\xi,\textbf{i})|>1.
\end{align*}
It follows from \eqref{(8.3)} that
\begin{align}\label{(8.5)}
|\hat\mu_{>((n+1)d)^k}(\xi+\frac 1 3 \iota^*(\textbf{i}))|\leq \gamma^{s(n,k,\xi,\textbf{i})},
\end{align}
where $\gamma=\max\{\frac13(1+|1+e^{-2\pi i y}|):{\|y\|\geq \frac1{\alpha\beta}}\}$.

{\bf Step 3. The key in the proof is to estimate the decay of $\gamma^{s(n,k,\xi,\textbf{i})}$ when $n$ is large enough.}
For simplicity, denote
\begin{align*}
X_j=|(\frac\beta\alpha)^jy(n,k,\xi,\textbf{i})|:=\alpha^{u_j}v_j+\varepsilon_j,\quad j=0,1,\cdots, r(n,k,\xi,\textbf{i}),
\end{align*}
where $u_j\in\mathbb{N}$, $v_j\in\mathbb{N}\setminus\alpha\mathbb{N}$ and $\varepsilon_j\in(-\frac12,\frac12]$. Let $q_1$ be the smallest positive integer such that $[\frac\alpha\beta]+1\leq \beta^{q_1}$.
 From Lemma \ref{lem4.4} again, we know that $j_{s(n,k,\xi,\textbf{i})}\geq r(n,k,\xi,\textbf{i})-q_1-1$. Hence
\begin{align}\label{(8.6)}
|(\frac\beta\alpha)^{q_1+3}X_{j_{s(n,k,\xi,\textbf{i})}}|\leq|(\frac\beta\alpha)^{2}X_{r(n,k,\xi,\textbf{i})}|<1.
\end{align}
We now establish the following recursive inequality,
\begin{align}\label{(8.7)}
X_{j_{i+1}}\geq(\frac\beta\alpha)^2X_{j_i}^{\frac{\log\beta}{\log\alpha}},\quad i=1,2,\cdots, s(n,k,\xi,\textbf{i})-1,
\end{align}
by considering the following two cases.

{\bf Case I}: $j_{i+1}=j_i+1$.  It is easy to see that
\begin{align*}
X_{j_{i+1}}=X_{j_i+1}>(\frac\beta\alpha)^{2}X_{j_i}>(\frac\beta\alpha)^{2}X_{j_i}^{\frac{\log\beta}{\log\alpha}}.
\end{align*}

{\bf Case II}: $j_{i+1}\geq j_i+2$. Then $|\varepsilon_{j_i+1}|<\frac1{\alpha\beta}$ and
\begin{align}\label{(8.8)}
||(\frac\beta\alpha)^{u_{j_i+1}+1}X_{j_i+1}||=||\frac{\beta^{u_{j_i+1}+1}v_{j_i+1}}\alpha+(\frac\beta\alpha)^{u_{j_i+1}+1}\varepsilon_{j_i+1}||\geq\frac1{\alpha\beta}.
\end{align}
The above inequality implies
\begin{align}\label{(8.9)}
j_{i+1}\leq {j_i+2+u_{j_i+1}}.
\end{align}
Note that
\begin{align}\label{(8.10)}
X_{j_i}=\frac\alpha\beta X_{j_i+1}=\frac\alpha\beta(\alpha^{u_{j_i+1}}v_{j_i+1}+\varepsilon_{j_i+1})\geq\alpha^{u_{j_i+1}}.
\end{align}
Combining \eqref{(8.9)} with \eqref{(8.10)}, we have
\begin{align*}
X_{j_{i+1}}\geq X_{j_i+2+u_{j_i+1}}\geq(\frac\beta\alpha)^{2+u_{j_i+1}}X_{j_i}\geq(\frac\beta\alpha)^{2}X_{j_i}^{\frac{\log\beta}{\log\alpha}-1}X_{j_i}=(\frac\beta\alpha)^{2}X_{j_i}^{\frac{\log\beta}{\log\alpha}}.
\end{align*}
Summarizing the above two cases, \eqref{(8.7)} follows.

Applying \eqref{(8.6)} and \eqref{(8.7)}, we obtain
\begin{align*}
(\frac\alpha\beta)^{q_1+3}\geq X_{j_{s(n,k,\xi,\textbf{i})}}&\geq (\frac \beta\alpha)^{2\sum_{n=1}^{s(n,k,\xi,\textbf{i})-2}(\frac{\log\beta}{\log\alpha})^n}X_{j_1}^{(\frac{\log \beta}{\log \alpha})^{s(n,k,\xi,\textbf{i})-1}}\geq \frac 1 {\alpha^2}X_{j_1}^{(\frac{\log \beta}{\log \alpha})^{s(n,k,\xi,\textbf{i})-1}}.
\end{align*}
Hence
\begin{align}\label{(8.11)}
j_1\log{\frac\beta\alpha}+\log{|y(n,k,\xi,\textbf{i})|}\leq (\log{\frac{\alpha^{q_1+5}}{\beta^{q_1+3}}})(\frac{\log \alpha}{\log \beta})^{s(n,k,\xi,\textbf{i})-1}.
\end{align}
Note that $j_1$ may be large enough when $n$ is large enough. In the following, we will prove that
\begin{align*}
j_1\leq2+\frac{\log |y(n,k,\xi,\textbf{i})|}{\log \alpha}.
\end{align*}
We only need to consider the case $j_1\geq 2$. Obviously,
\begin{align*}
X_j=(\frac\beta\alpha)^{j-1}X_1=\frac{\beta^{j-1}v_1}{\alpha^{j-1-u_1}}+(\frac\beta\alpha)^{j-1}\varepsilon_1,\quad j\geq2.
\end{align*}
Since $|\varepsilon_1|<\frac 1{\alpha\beta}, \frac{q^{d}}{p^{d}}=\frac{\beta}{\alpha}$ and $p>q$, we know $|\varepsilon_j|=|(\frac\beta\alpha)^{j-1}\varepsilon_1|<\frac1{\alpha\beta}$ for  $j\leq u_1+1$. Replacing $j_i$ by $0$ in \eqref{(8.8)}, we obtain
$|\varepsilon_{u_1+2}|\geq\frac 1{36}$. This implies $j_1=u_1+2$.  Similar to \eqref{(8.10)}, we have $u_1\log \alpha\leq\log |y(n,k,\xi,\textbf{i})|$. Hence the desired result follows, and
\begin{align*}
j_1\log{\frac\beta\alpha}+\log{|y(n,k,\xi,\textbf{i})|}&\geq\frac{\log \beta}{\log \alpha}\log{|y(n,k,\xi,\textbf{i})|}+2\log{\frac\beta\alpha}.
\end{align*}
This inequality, \eqref{(8.4)} and \eqref{(8.11)} show that there exists an integer $n_1$(depends only on $k$) such that
\begin{align}\label{(8.12)}
(\frac{\log \beta}{\log \alpha})^{s(n,k,\xi,\textbf{i})}\leq \frac {\log{\frac{\alpha^{q_1+5}}{\beta^{q_1+3}}}} {\log{|y(n,k,\xi,\textbf{i})|}}\leq\frac{\log{\frac{\alpha^{q_1+5}}{\beta^{q_1+3}}}}{n^k\log q}:=\frac{c_1}{n^k},\quad n\geq n_1.
\end{align}

{\bf Final step}.
Put $\zeta=\log \gamma/\log\frac{\log\beta}{\log\alpha}$. Combining \eqref{(8.5)} with \eqref{(8.12)}, we obtain
\begin{align*}
|\hat\mu_{>(d(n+1))^k}(\xi+\frac 1 3 \iota^*(\textbf{i}))|\leq {(\frac{\log\beta}{\log\alpha})^{\zeta{s(n,k,\xi,\textbf{i})}}}\leq \frac{c_1^{\zeta}}{n^{k\zeta}},\quad n\geq n_1,
\end{align*}
for any $\xi\in[0,\frac13]$, $k\in\Bbb N$ and $\mathbf{i}\in I^k_{n,n+1}$. Therefore
\begin{align*}
Q_{n+1,k}(\xi)&= Q_{n,k}(\xi)+\sum_{\textbf{i}\in I^k_{n,n+1}}|\hat\mu(\xi+\frac 1 3 \iota^*(\textbf i))|^2\nonumber\\
&=Q_{n,k}(\xi)+\sum_{\textbf{i}\in I^k_{n,n+1}}|\hat\mu_{(d(n+1))^k}(\xi+\frac 1 3 \iota^*(\textbf i))|^2\cdot |\hat\mu_{>(d(n+1))^k}(\xi+\frac 1 3 \iota^*(\textbf i))|^2\nonumber\\
&\leq Q_{n,k}(\xi)+\frac{c_1^{2\zeta}}{n^{2k\zeta}}\sum_{\textbf{i}\in I^k_{n,n+1}}|\hat\mu_{(d(n+1))^k}(\xi+\frac 1 3 \iota^*(\textbf i))|^2\nonumber\\
&\leq Q_{n,k}(\xi)+\frac{c_1^{2\zeta}}{n^{2k\zeta}}(1-\sum_{\textbf{i}\in I^k_n}|\hat\mu_{(d(n+1))^k}(\xi+\frac 1 3 \iota^*(\textbf i))|^2)\quad (\text{by Lemma \ref{lem4.3}})\nonumber\\
&\leq Q_{n,k}(\xi)+\frac{c_1^{2\zeta}}{n^{2k\zeta}}(1-Q_{n,k}(\xi)).
\end{align*}
for $n\geq n_1$. By recursion, we conclude that
\begin{align*}
1-Q_{n+1,k}(\xi)\geq (1-Q_{n_1,k}(\xi))\prod_{j=n_1}^{n}(1-\frac{c_1^{2\zeta}}{n^{2k\zeta}}).
\end{align*}
Taking $k=[\frac 1 {\zeta}]+1$, we see from the above inequality that
\begin{align*}
1-Q(\xi)\geq C(1-Q_{n_1,k}(\xi)),
\end{align*}
where $C=\prod_{j=n_1}^{\infty}(1-\frac{c_1^{2\beta}}{n^{2k\beta}})>0$. Hence $Q(\xi)\not\equiv1$ on $[0,\frac13]$.  From Theorem \ref{lem2.1}, $\Lambda$ is not a spectrum of $\mu_{\{p_n\},\{\mathcal{D}_n\}}$. The proof is complete.
\end{proof}

At the end of this section, we compare our Theorem \ref{thm1.5} with \cite[Theorem 1.4]{AHH19}.
\begin{thm}\cite[Theorem 1.4 or Lemma 3.5]{AHH19}\label{thm4.5}
Let $\{q_n,p_n\}_{n=1}^{\infty}$ be a sequence of positive integers satisfying that $q_n<p_n\leq M$ and $\frac{q_n\cdots q_m}{\gcd (p_n\cdots p_m, q_n\cdots q_m)}\geq 2^{m-n+1}$ for some positive number $M$ and all $m\geq n\geq 1$. Then, for any co-prime integers $a_n,b_n$ so that $q_1q_2\cdots q_n\in\{a_n,b_n\}$, the $\mu_{\{p_n\},\{0,a_n,b_n\}}$ is not a spectral measure.
\end{thm}

If we assume $q_n=q$ and $p_n=p$, Theorem \ref{thm4.5} becomes the following result.

\begin{thm}\cite{AHH19}\label{thm4.6}
Let $p,q\in \Bbb N^+$ with $p>q\geq 2\gcd (p,q)$. Then, for any co-prime integers $a_n,b_n$ so that $q^n\in\{a_n,b_n\}$, the $\mu_{p,\{0,a_n,b_n\}}$ is not a spectral measure.
\end{thm}

Clearly, our Theorem \ref{thm1.5} only need the conditions $p>q\geq 2$ and $c\cdot q^{\omega_n}\in\{a_{\omega_n},b_{\omega_n}\}$ which are both weaker than \cite[Theorem 1.4]{AHH19}.
Applying Theorem \ref{thm1.5} directly, we have the following non-spectral example.
\begin{ex}
Let $a_n=1$,\;$b_n=\begin{cases}
2^n, & n \text{ is odd,}\\
2^{n-1}, & n \text{ is even.}\end{cases}$, $p_n=3$ for $n\geq 1$. Then $\mu_{3,\{0,1,b_n\}}$ is a Moran measure, but not a spectral measure.
\end{ex}

\end{document}